\documentclass[11pt,a4paper]{article}
\usepackage{amsmath,amssymb,amscd,amsthm}
\usepackage{graphics}
\input xy
\usepackage{comment}
\usepackage{tikz}
\usetikzlibrary{arrows}
\usepackage{enumerate}
\usepackage[paperwidth=210mm,paperheight=297mm,a4paper,twoside,top=1in, left=1.2in, right=1.2in, bottom=1in]{geometry}
\usepackage[active]{srcltx}
\usepackage{float}
\usepackage[all]{xy}
\usepackage[dvips]{lscape}
\xyoption{web}
\restylefloat{figure}

\newtheorem{lemma}{Lemma}

\newtheorem{prop}[lemma]{Proposition}
\newtheorem{cor}[lemma]{Corollary}

\newtheorem{thm}[lemma]{Theorem}

\newtheorem{rmk}[lemma]{Remark}

\newcommand{\uh}{\mathcal{U}(\mathfrak{h})}

\newcommand{\bb}[1]{\mathbb{#1}}

\newcommand{\ca}[1]{\mathcal{#1}}

\newcommand{\ra}{\rightarrow}

\usepackage[parfill]{parskip}

\begin{document}

\title{New family of simple  $\mathfrak{gl}_{2n}(\bb{C})$-modules}
\author{Jonathan Nilsson}
\date{}
\maketitle

\begin{abstract}
\noindent
We construct a new family of simple $\mathfrak{gl}_{2n}$-modules which depends on
$n^2$ generic parameters. Each such module is isomorphic to the regular 
$\ca{U}(\frak{gl}_{n})$-module when restricted the $\mathfrak{gl}_{n}$-subalgebra
naturally embedded into the top-left corner.
\end{abstract}

\section{Introduction}
Classification of simple modules is one of the first natural questions which arises
when studying the representation theory of some (Lie) algebra. Simple modules are,
in some sense, ``building blocks'' for all other modules, and hence understanding simple
modules is important. In some cases, for example for finite dimensional associative 
algebras, classification of simple modules is an easy problem. However, in 
most of the cases, the problem of classification of all simple
modules is very difficult. Thus, if we consider simple, finite dimensional,
complex Lie algebras, then the only algebra for which some kind of classification exists 
is the Lie algebra $\frak{sl}_{2}$. This was obtained by R.~Block in \cite{Bl}, see 
also a detailed explanation in \cite[Chapter~6]{Ma}. However, even in this case the ``answer''
only reduces the problem to classification of equivalence classes of irreducible elements
in a certain non-commutative Euclidean ring.

At the moment, the problem of {\em classification} of simple modules over simple Lie
algebras seems too hard. However, because of its importance, the problem of {\em construction}
of new families of modules attracted a lot of attention over the years. The most studied
case seem to be the one of the Virasoro Lie algebras, where many different multi-parameter
families of simple modules were constructed by various authors, see, for example, 
\cite{OW,GLZ,LZ,LLZ,MZ1,MZ2,MW} and references therein.

In contrast to the Virasoro case, the ``easier'' case of simple, complex, finite dimensional
Lie algebras does not yet have an equally large variety of families of simple modules.
So, let $\mathfrak{g}$ be a complex, finite dimensional, simple Lie algebra.
Some classes of simple $\mathfrak{g}$-modules are, of course, well-understood. For example:
\begin{itemize}
\item simple {\em finite dimensional} modules are classified already by Cartan in 1913,
see \cite{Ca};
\item simple {\em highest weight} modules related to a fixed triangular decomposition 
$\frak{n}_{-} \oplus \frak{h} \oplus \frak{n}_{+}$ of $\mathfrak{g}$ are classified
by their highest weights and are extensively studied during last 50 years, see, for example,
\cite{Dix,Hum2,BGG};
\item simple Whittaker modules in the sense of \cite{Ko}, see also \cite{AP,McD1,McD2};
\item simple {\em Gelfand-Zeitlin} modules, see \cite{DFO,DFO2,Ma0,FGR};
\item simple weight modules with {\em finite dimensional weight spaces} were classified in 
\cite{Mathieu} extending the previous work in \cite{Fe,Fu}; 
\item simple $\mathfrak{g}$-modules which are free of rank one over the universal enveloping
algebra of the Cartan subalgebra were constructed and studied in \cite{Ni1,Ni2}
(see also \cite{TZ1,TZ2} for similar modules over infinite dimensional Lie algebras).
\end{itemize}
Some further classes of simple modules can be found in \cite{FOS}.
We note that the largest known family of simple $\mathfrak{gl}_{n}$-modules is
the one of Gelfand-Zeitlin-modules. It depends on $\frac{n(n+1)}{2}$ generic complex parameters, 
see \cite{DFO,DFO2} for details.

Based on the above, it seems natural to look for new families of simple $\mathfrak{g}$-modules.
The present paper contributes with a new large family of simple $\mathfrak{gl}_{2n}$-modules.
This family is parameterized by invertible $n\times n$ complex matrices.
Let $\ca{A,B,C,D}$ be the four Lie subalgebras of $\frak{gl}_{2n}$ of dimension $n^2$
as indicated in the following figure:
\[\left( \begin{array}{cc}
   \ca{A} & \ca{B}\\
     \ca{C} & \ca{D}
  \end{array} \right).\]
Then $\ca{B}$ is nilpotent (and even commutative), and the adjoint action of $\ca{B}$ 
on $\frak{gl}_{2n} / \ca{B}$ is nilpotent, so $(\ca{B},\frak{gl}_{2n})$
is a \emph{Whittaker pair} in the sense of~\cite{BM}.
The original motivation for this paper was an attempt to describe 
generalized Whittaker modules (i.e. modules on which the action of $\ca{B}$ is locally finite) 
for this Whittaker pair. Our main result can be summarized as follows:

\begin{thm}\label{introthm}
For each non-degenerate complex $n \times n$-matrix $Q$, there exists a simple 
$\mathfrak{gl}_{2n}$ module $M$ with the following properties:
\begin{itemize}
\item $M$ has Gelfand-Kirillov dimension $n^{2}$;
\item $Res_{\ca{A}}^{\frak{gl}_{2n}}M$ is isomorphic to the left regular $\ca{U(A)}$-module;
\item $Res_{\ca{B}}^{\frak{gl}_{2n}}M$ is locally finite. In other words, $M$ is a 
generalized Whittaker module for the Whittaker pair $(\ca{B},\frak{gl}_{2n})$;
\item With respect to a fixed PBW basis in $\ca{U(A)}$, the action of each fixed element from 
$\ca{A,B,C,D}$ can be written explicitly as maps $\ca{U(A)} \rightarrow \ca{U(A)}$ of 
degrees $1,0,2,1$, respectively.
\end{itemize}
Moreover, different matrices $Q$ give non-isomorphic modules.
\end{thm}

The paper is organized as follows. 
Section 2 introduces notation and lays down some motivation for the construction of our modules. 
In the same section, for each non-degenerate complex $n \times n$-matrix $Q$, we construct an 
$\ca{A+B}$-module having the first three properties listed in Theorem~\ref{introthm}.
We show that there must exist a simple quotient of the corresponding induced $\frak{gl}_{2n}$ 
module that also has the fourth property. In Section 3 we explicitly construct such a module for $Q$
being the identity matrix $I$ and show that every other module in our family can be obtained by 
twisting this module by an explicit automorphism. Finally, we give explicit formulas for 
the $\frak{gl}_{2n}$-action in all cases.

\noindent
{\bf Acknowledgements}
I am very grateful to Volodymyr Mazorchuk for his ideas and comments.

\section{Motivation and existence}
\subsection{Setup}
Let $\mathfrak{g}:=\mathfrak{gl}_{2n}(\bb{C})$. All Lie algebras and vector spaces are over the complex numbers. $\bb{N}$ denotes the set of nonnegative integers.

First we observe that the subalgebras $\ca{A}$ and $\ca{D}$ defined above are both
isomorphic to $\frak{gl}_{n}$ while the subalgebras $\ca{B}$ and $\ca{C}$ are commutative.
Let $e_{i,j}$ be the $2n \times 2n$-matrix with a single $1$ in position $(i,j)$ and 
zeros elsewhere. By convention, most indices $i,j$ etc. can be assumed to lie between $1$ and $n$;
 in particular our canonical basis for $\frak{gl}_{2n}$ will be written
\[\bigcup_{1 \leq i,j \leq n} \{e_{i,j},e_{n+i,j},e_{i,n+j},e_{n+i,n+j}\}.\]
We denote the identity matrix by $I$, its size ($n$ or $2n$) should be apparent by the context. The transpose of a matrix $A$ is denoted $A^{T}$ and if $A$ is invertible we abbreviate $(A^{-1})^{T}$ by $A^{-T}$.

We also recall how to construct \emph{twisted modules}. For every Lie algebra automorphism $\varphi \in Aut(\frak{g})$ we have a twisting functor
 $F_{\varphi}: \frak{g}\text{-mod} \rightarrow \frak{g}\text{-mod}$ which is an auto-equivalence. It maps a module $M$ to ${}^{\varphi}M$ which is isomorphic to $M$ as a vector space
 but has modified action: $x \bullet v := \varphi(x) \cdot v$ for all $x \in \frak{g}$ and $v \in {}^{\varphi}M$.

\subsection{Existence of simple generalized Whittaker Modules for $\mathfrak{gl}_{2n}$}
Following Kostant's idea in~\cite{Ko} we try to construct some modules on which the action of $\ca{B}$ is locally finite.

Fix Lie algebra homomorphisms $\lambda_{A}: \ca{A} \rightarrow \bb{C}$ and $\lambda_{D}: \ca{D} \rightarrow \bb{C}$.
Let $\bb{C}_{\lambda_{A}\lambda_{D}}$ be the one dimensional $(\ca{A} + \ca{C} + \ca{D})$-module where $\ca{A}$ acts by $\lambda_{A}$, $\ca{D}$ acts by $\lambda_{D}$ and $\ca{C}$ acts trivially.
Now define a generalized Verma module 
\[M_{\lambda_{A}\lambda_{D}} := U(\mathfrak{gl}_{2n}) \bigotimes_{U(\ca{A} + \ca{C} + \ca{D})}\bb{C}_{\lambda_{A}\lambda_{A}}.\]
Denote by $M_{\lambda_{A}\lambda_{D}}^{*}$ the full dual of $M_{\lambda_{A}\lambda_{D}}$. This is a $\mathfrak{gl}_{2n}$ module where the action is given $(x \cdot f)(m) = -f(x\cdot m)$ as usual.

\begin{prop}
 For every $\theta: \ca{B} \rightarrow \bb{C}$, there is a unique (up to multiple) eigenvector $w$ in $M_{\lambda_{A}\lambda_{D}}^{*}$ with eigenvalue $\theta$ for $\ca{B}$.
\end{prop}
\begin{proof}
Note that $M_{\lambda_{A}\lambda_{D}} \simeq U(\ca{B})$ as a left and right $U(\ca{B})$-module. Let $\bb{C}(\theta)$ be the $1$-dimensional $\ca{B}$-module where the action is given by $\theta$.
By the tensor-hom adjunction we have
\begin{align*}
 Hom_{U\ca{(B)}}(\bb{C}(\theta),M_{\lambda_{A}\lambda_{D}}^{*}) &=  Hom_{U\ca{(B)}}(\bb{C}(\theta),Hom_{\bb{C}}\big( M_{\lambda_{A}\lambda_{D}},\bb{C}) \big)\\
&\simeq Hom_{U\ca{(B)}}(\bb{C}(\theta),Hom_{\bb{C}}\big( U\ca{(B)},\bb{C}) \big)\\
& \simeq Hom_{\bb{C}}(U\ca{(B)} \otimes_{U\ca{(B)}}\bb{C}(\theta),\bb{C})\\
& \simeq Hom_{\bb{C}}(\bb{C}(\theta),\bb{C}) \simeq \bb{C}.\\
\end{align*}
Thus there is a unique $1$-dimensional subspace of $M_{\lambda_{A}\lambda_{D}}^{*}$ isomorphic to $\bb{C}(\theta)$ in $\ca{B}$-mod, which is equivalent to the statement in the proposition.
\end{proof}
The submodule generated by such an eigenvector must be simple (see~\cite{BM}), so we get the following result.
\begin{cor}
 There exist simple generalized Whittaker modules for the pair $(\ca{B},\frak{gl}_{2n})$ and they can be realized as simple submodules in the dual of the generalized Verma module $M_{\lambda_{A}\lambda_{D}}^{*}$.
\end{cor}
The drawback with this approach in our case is that it is difficult to say anything more explicit about the resulting modules as $M_{\lambda_{A}\lambda_{D}}^{*}$ is very big and inconvenient to work in.

\subsection{An $\ca{A+B}$-module}
\subsubsection{Construction and a formula for the action}
We now turn to a more explicit construction. Note that $\ca{B}$ is commutative. Let $Q=(q_{ij})$ be a nonsingular $n\times n$ matrix and define $L_{Q}$ to be the $1$-dimensional $U(\ca{B})$-module with generator $v$ where
the action of $\ca{B}$ is given by $Q$: \[ e_{i,n+j} \cdot v := q_{i,j}v \qquad 1 \leq i,j \leq n.\] 

Define an induced module
\[M_{Q}:= Ind_{\ca{B}}^{\ca{A+B}} L_{Q} = U(\ca{A}+\ca{B}) \bigotimes_{U(\ca{B})} L_{Q}.\]
Then $M_{Q}$ is clearly isomorphic to $U\ca{(A)}$ as a left $\ca{A}$-module, and for $a \in U\ca{(A)}$ we shall write just $av$ or just $a$ for $a \otimes v$.
To explicitly see how $\ca{B}$ acts on $M_{Q}$, we introduce some more notation.
Consider $U\ca{(A)} \otimes_{\bb{C}} \ca{A}$ as a tensor product in the category of unital associative algebras.  This becomes an infinite dimensional Lie algebra under the commutator bracket. 
Note that $U\ca{(A)} \otimes \ca{A} \simeq Mat_{n \times n}(\ca{U(A)})$ in a natural way and we shall even extend the trace function to $U\ca{(A)} \otimes \ca{A}$ by defining $tr(a \otimes B):=a \:tr(B)$.
Note also that $\ca{A}$ embeds into $U\ca{(A)} \otimes \ca{A}$ (as both associative algebra and Lie algebra) by the map $A \mapsto 1 \otimes A$, and we shall sometimes need to identify
elements of $\ca{A}$ with their images under this map.
To resolve some ambiguity in our notation, for $A,B \in \ca{A}$ we shall write $AB$ for the product in $\ca{U(A)}$ and $A.B$ for the product in the associative algebra $\ca{A}$ or $U\ca{(A)} \otimes \ca{A}$.

Let $\psi': \ca{A} \rightarrow U\ca{(A)} \otimes_{\bb{C}} \ca{A}$ be the Lie algebra homomorphism defined by 
\[\psi': A \mapsto A \otimes I -1 \otimes A^{T}.\]
This extends to an algebra homomorphism $\psi: U\ca{(A)} \rightarrow U\ca{(A)} \otimes_{\bb{C}} \ca{A}$.

\begin{lemma}
The action of $\ca{B}$ on $M_{Q}$ is given by
\[\left( \begin{array}{cc}0 & B\\0 & 0\end{array} \right) av = tr(\psi(a).Q.B^{T})v.\]
\end{lemma}
\begin{proof}
 This follows by induction on the degree of $a$ as follows. The lemma clearly holds for $a=1$ by the definition of the action of $\ca{B}$ on $L_{Q}$: we have $tr(Q.B^{T}) = \sum_{ij}q_{ij}b_{ij}$.
 Suppose the lemma holds for all monomials $a$ of a fixed degree (with respect to any fixed PBW basis). 
We then have
\begin{align*}
 \left( \begin{array}{cc}0 & B\\0 & 0\end{array} \right) (Aa)v &=  A \left( \begin{array}{cc}0 & B\\0 & 0\end{array} \right) av 
+ \big[ \left( \begin{array}{cc}0 & B\\0 & 0\end{array} \right), \left( \begin{array}{cc}A & 0\\0 & 0\end{array} \right) \big]av\\
&=A \left( \begin{array}{cc}0 & B\\0 & 0\end{array} \right) av -  \left( \begin{array}{cc}0 & A.B\\0 & 0\end{array} \right)av\\
&=A \: tr(\psi(a).Q.B^{T})v -  tr(\psi(a).Q.(A.B)^{T})v\\
&=tr((A \otimes I).\psi(a).Q.B^{T})v -  tr(A^{T}.\psi(a).Q.B^{T})v\\
&=tr(((A \otimes I)-1 \otimes A^{T}).\psi(a).Q.B^{T})v\\
&=tr(\psi(A).\psi(a).Q.B^{T})v\\
&=tr(\psi(Aa).Q.B^{T})v.\\
\end{align*}
This shows that the lemma holds for all monomials in $\ca{U(A)}$ by induction. Since $\psi$ is linear it holds for all of $\ca{U(A)}$.
\end{proof}

\subsubsection{Proof of simplicity}
We proceed to prove that $M_{Q}$ is simple by first proving it for $Q=I$.

\begin{lemma}
\label{rel}
 The following relations hold in $U\ca{(A+B)}$.
\[
[e_{j,k+n},e_{i,j}^{m}] = \begin{cases}
                         -m \: e_{i,j}^{m-1} e_{i,k+n} & \text{ for } i \neq j\\
			 ((e_{i,j}-1)^{m} -e_{i,j}^{m})e_{i,k+n} & \text{ for } i = j.  \\
                        \end{cases}
\]
\end{lemma}
\begin{proof}
 This follows easily by induction on $m$.
\end{proof}
Fix a PBW basis of $U(\ca{A})$ of form
\[\{ e_{11}^{l_{11}}e_{12}^{l_{12}} \cdots e_{1n}^{l_{1n}} e_{21}^{l_{21}} \cdots \cdots  e_{n1}^{l_{n1}} \cdots e_{nn}^{l_{nn}}   | l_{ij} \in \bb{N}\},\]

Then $U(\ca{A}) \simeq M_{I}$ has a filtration:
\[ M_{I}^{(0)} \subset M_{I}^{(1)} \subset M_{I}^{(2)} \subset \cdots \]
where $M_{I}^{(m)}$ is the span of all monomials $f$ with $\deg f := \sum_{ij}l_{ij} \leq m$. 

\begin{lemma}
\label{mod}
 For each $1 \leq j,k \leq n$, the element $(e_{j,k+n} - \delta_{j,k}) \in U(\ca{B})$ has degree $-1$ with respect to the filtration of $M_{I}$.
 Moreover, the action on an arbitrary monomial in  $M_{I}^{(d)}$ is given by
\[(e_{j,k+n} - \delta_{j,k}) \cdot  e_{11}^{l_{11}} \cdots e_{kj}^{l_{kj}} \cdots  e_{nn}^{l_{nn}} = -l_{kj} \: e_{11}^{l_{11}} \cdots e_{kj}^{l_{kj}-1} \cdots  e_{nn}^{l_{nn}} \mod M_{I}^{(d-2)}.\]
\end{lemma}
\begin{proof}
 We have \[(e_{j,k+n} - \delta_{j,k}) \cdot f = f (e_{j,k+n} - \delta_{j,k}) + [e_{j,k+n} - \delta_{j,k},f] = [e_{j,k+n},f],\]
so the fact that $(e_{j,k+n} - \delta_{j,k})$ has degree $ \leq -1$ follows from the previous lemma and the fact that $ad_{e_{j,k+n}}$ is a derivation.

For the second more precise statement, let $f$ be an arbitrary monomial of degree $d$. For each $i$ let $P_{i}$,$Q_{i}$
 be the monomial factors of $f$ such that $f= P_{i} e_{ij}^{l_{ij}} Q_{i}$ and $e_{ij} \not| P_{i},Q_{i}$. We now calculate
\begin{align*}
 (e_{j,k+n} &- \delta_{j,k}) \cdot f = [e_{j,k+n},f] = \sum_{i}P_{i} [e_{j,k+n},e_{ij}^{l_{ij}}] Q_{i}\\
 &=   P_{j} ((e_{jj}-1)^{l_{jj}}-e_{jj}^{l_{jj}}) e_{j,k+n}\cdot Q_{j}    + \sum_{i \neq j}-l_{ij} \: P_{i} e_{ij}^{l_{ij}-1} e_{i,k+n}\cdot Q_{i}
\end{align*}

By writing $e_{i,k+n} = (e_{i,k+n}-\delta_{ik})+\delta_{ik}$ and using the fact that the first term has negative degree, we see that
\begin{align*}
 (e_{j,k+n} - \delta_{j,k}) \cdot f &=  \delta_{j,k} P_{j} ((e_{jj}-1)^{l_{jj}}-e_{jj}^{l_{jj}}) Q_{j}    + \sum_{i \neq j}- \delta_{ik} l_{ij} \: P_{i} e_{ij}^{l_{ij}-1} Q_{i} \mod M_{I}^{(d-2)}\\
 &=  -\delta_{j,k}l_{jj} \: P_{j} e_{ij}^{l_{ij}-1} Q_{j}   + \sum_{i \neq j}- \delta_{ik} l_{ij} \: P_{i} e_{ij}^{l_{ij}-1} Q_{i} \mod M_{I}^{(d-2)}\\
&=  -\sum_{i}\delta_{ik}l_{ij} \: P_{i} e_{ij}^{l_{ij}-1} Q_{i} \mod M_{I}^{(d-2)}\\
&=  -l_{kj} \: P_{k} e_{kj}^{l_{kj}-1} Q_{k} \mod M_{I}^{(d-2)}.
\end{align*}
The lemma follows.
\end{proof}

\begin{cor}
 For each $1 \leq i,j \leq n$, the action of $(e_{i,j}-\delta_{i,j})$ on $M_{I}$ is surjective. Its kernel is spanned by all monomials not divisible by $e_{ij}$.
\end{cor}

\begin{prop}
\label{sprop}
 The module $M_{I}$ is simple in $U(\ca{A+B})$-mod.
\end{prop}
\begin{proof}
It suffices to show that any $f \in M_{I}$ can be reduced to $1 \in M_{I}^{0}$ via the $\ca{B}$-action. Fix $f \in M_{I}$ and let $p \in M_{I}^{(d)}$ be a nonzero monomial occurring in $f$ with maximal degree $d$.
If $p= \prod_{ij}e_{ij}^{l_{ij}}$ (in the PBW order), it is clear by the previous lemma that $B_{p}:=\prod_{ij}(e_{j,n+i}-\delta_{ij})^{l_{ij}} \in U(\ca{B})$ maps $p$ to a nonzero constant. By the maximality of
 $d$, $B_{p}$ annihilates all other monomials occurring in $f$ so in fact $B_{p} \cdot f \in M_{I}^{(0)}$ is a nonzero constant as desired.  
\end{proof}

\begin{cor}
The module $M_{Q}$ is simple if and only if $Q$ is nonsingular.
\end{cor}
\begin{proof}
For each nonsingular $S \in \ca{A}$, define $\varphi_{S}: \ca{A+B} \rightarrow \ca{A+B}$ by
\[\varphi_{S}: \left( \begin{array}{cc} A & B \\ 0 & 0  \end{array}\right) \mapsto \left( \begin{array}{cc} A & B.S^{-1} \\ 0 & 0  \end{array}\right).\]
It is easy to verify that $\varphi_{S}$ is a Lie algebra automorphism and that $\varphi_{S} \circ \varphi_{T} = \varphi_{ST}$.
It is also clear that the twisted module ${}^{\varphi_{Q^{-T}}}M_{I}$ is isomorphic to $M_{Q}$.
Since $M_{I}$ is simple by Proposition~\ref{sprop}, and since twisting by automorphisms defines an auto-equivalence on $\mathfrak{gl}_{2n}$-Mod, $M_{Q}$ is also simple for nonsingular $Q$.

Conversely, assume that $Q$ is singular and let $A$ be a nonzero matrix such that $Q^{T}A=0$. We shall show that $U\ca{(A)}Av$ is a proper $\ca{A+B}$-submodule of $M_{Q}$.
The subspace $U\ca{(A)}Av$ is clearly $\ca{A}$-stable. For $a \in U\ca{(A)}$ we compute
\begin{align*}
\left( \begin{array}{cc} 0 & B \\ 0 & 0\end{array} \right) \cdot aAv &= tr(\psi(aA).Q.B^{T})v\\
&= tr(\psi(a).\psi(A).Q.B^{T})v = tr(\psi(a).(A \otimes I -1\otimes A^{T}).Q.B^{T})v\\
&=tr(Q.B^{T}.\psi(a).(A \otimes I))v- tr(\psi(a).A^{T}.Q.B^{T})v \\
&= tr(Q.B^{T}.\psi(a))Av-tr(\psi(a).(Q^{T}.A)^{T}.B^{T})v \\
&=   tr(Q.B^{T}.\psi(a))Av.
\end{align*}
Thus $U\ca{(A)}Av$ is also $\ca{B}$-stable, and is thus a proper submodule of $M_{Q}$.
\end{proof}

\subsubsection{Injectivity and an existence theorem}
Our next goal is to prove that for most $Q$'s, the module $M_{Q}$ is injective when restricted to $\ca{U(B)}$.
We begin by recalling a result about injective envelopes for the trivial module over polynomial rings. For a proof, see for example~\cite[$\S 3$J]{L}.
\begin{lemma}
 Let $k$ be a field, let $R=k[x_{1}, \ldots x_{n}]$ and let $L$ be the trivial $R$-module. Let $E$ be the $R$-module $k[x_{1}^{-1}, \ldots x_{n}^{-1}]$ where $x_{i}$ acts by
\[x_{i} \cdot (x_{1}^{-k_{1}} \cdots x_{n}^{-k_{n}}) = \begin{cases}
                                                     x_{1}^{-k_{1}} \cdots x_{i}^{-k_{i}+1} \cdots x_{n}^{-k_{n}} & \text{if } k_{i} > 0 \\
						     0                                                         & \text{otherwise.}
                                                   \end{cases} \]
Then $E=E(L)$ is the injective envelope of $L$.
\end{lemma}
By twisting $E$ by automorphisms we obtain injective envelopes of all $1$-dimensional $R$-modules as follows:

\begin{cor}
With notation as in the previous lemma, for scalars $q_{i} \in k$,
let $L_{q_{1}, \ldots q_{n}}$ be the $1$-dimensional $R$-module with action $x_{i} \cdot v = q_{i} v$. 
Then $E(L_{q_{1}, \ldots q_{n}})  \simeq {}^\varphi E(L)$ where $\varphi$ is the $R$-automorphism mapping $x_{i} \mapsto x_{i}-q_{i}$. 
\end{cor}
\begin{proof}
 We have $L_{q_{1}, \ldots q_{n}} \simeq {}^\varphi L$ and since twisting by an automorphism is an auto-equivalence on $R$-mod, the corollary follows.
\end{proof}

\begin{prop}
 For nonsingular matrices $Q$, the module $Res_{U(\ca{B})}^{U(\ca{A+B})} \; M_{Q}$ is injective.
\end{prop}
\begin{proof}
 Let $I(L_{Q})$ be the injective envelope of $L_{Q}$. Applying the exact functor $Hom_{\ca{B}}( -,I(L_{Q}))$ to the exact sequence 
\[0 \ra L_{Q} \ra M_{Q} \ra Coker \ra 0\]
we obtain the exact sequence
\[0 \ra Hom_{\ca{B}}( Coker,I(L_{Q})) \ra Hom_{\ca{B}}( M_{Q},I(L_{Q})) \ra Hom_{\ca{B}}(L_{Q},I(L_{Q})) \ra 0.\]
Hence the morphism $L_{Q} \ra I(L_{Q})$ mapping $L_{Q}$ into its injective envelope is the image of some morphism $f:M_{Q} \ra I(L_{Q})$. Since $f$ is nonzero on $span(v)=soc(M_{Q})$,
$f$ is injective. Moreover, for all $k \in \bb{N}$ we have \[\dim soc_{k}(M_{Q}) = \binom{n^{2}+k-2}{k-1}= \dim soc_{k}(I(L_{Q})),\]
which shows that $f$ is surjective. This shows that $f$ is an isomorphism and in particular that $M_{Q}$ is the injective envelope of $L_{Q}$.
\end{proof}

\begin{rmk}
 Indecomposable injectives over noetherian rings R correspond to Spec(R) via $\frak{p} \mapsto $injective envelope of $(R/ \frak{p} )$.
Moreover $L_{Q} = U(\ca{B}) / \mathfrak{m}$ where $\frak{m}$ is the maximal ideal generated by $(e_{i,n+j}-q_{i,j})$, so if $M_{Q}$ is injective, it must be the injective envelope of $U(\ca{B}) / \mathfrak{m}$.
\end{rmk}

\begin{thm}
 For each nonsingular matrix $n \times n$-matrix $Q$ there exists a $\mathfrak{gl}_{2n}$-module $M$ such that
\begin{itemize}
 \item $M$ is generated by a single $\ca{B}$-eigenvector with eigenvalues corresponding to the entries of $Q$.
 \item $Res_{U(\ca{B})}^{U({\mathfrak{gl}_{2n}})} M \simeq U(\ca{A}) \simeq U({\mathfrak{gl}_{n}})$.
\end{itemize}
\end{thm}
\begin{proof}
 As we've seen before, we take $L_{Q}$ as the $1$-dimensional $\ca{B}$-module corresponding to $Q$ and we let $M_{Q} = U(\ca{A}+\ca{B}) \bigotimes_{U(\ca{B})} L_{Q}$.
 Then $M_{Q}$ is injective in $\ca{B}$-mod. Next we define
 \[W:= U(\ca{A}+\ca{B}+\ca{D}) \bigotimes_{U(\ca{A}+\ca{B})} M_{Q}.\]
Fixing $d \in \ca{D}$ we note that $span(v,d\cdot v)$ is a two-dimensional $\ca{B}$-submodule of $W$, and moreover it is a non-split self-extension of $L_{Q}$ with itself.
Now by the injectivity of $M_{Q}$ there exists a morphism $\varphi$ such that the following diagram commutes in $\ca{B}$-mod:
\[\xymatrix@C=2cm@R=2cm{  & span(v,d \cdot v) \ar@{.>}[dl]^{\varphi} \\ M_{Q} & L_{Q} \ar@{^{(}->}[u] \ar@{^{(}->}[l]}\]

Thus there exists $a_{d} \cdot v \in soc_{2}(M_{Q})=\ca{A} \cdot v$ such that $a_{d}\cdot v-d\cdot v$ spans a $1$-dimensional $B$-submodule $S_{d}$ of $W$. The module
 $W' := W / \sum_{d \in \ca{D}} U(\ca{A+B+D})S_{d}$ is then isomorphic to $M_{Q}$ when restricted to $U(\ca{A+B})$.

Next, let $W'' := U(\ca{A+B+C+D}) \bigotimes_{U(\ca{A+B+D})} W'$.
For a fixed $c \in \ca{C}$ we have a $\ca{B}$-submodule $\ca{B}^{2} (c\cdot v)$ with simple top and simple socle, both isomorphic to $L_{Q}$. By similar arguments, there exists 
 $x \in soc_{3}(M_{Q}) = \ca{A}^{2} \cdot v$ such that $x-c \cdot v$ spans a $\ca{B}$-submodule of $W''$. Forming the quotient of all these subs we get the module required by the theorem.
\end{proof}

In the next section we shall give explicit formulas for the elements $a_{d}$ and $x$ of the proof above in order to write down the action on the simple $\frak{gl}_{2n}$-modules explicitly.

\section{Explicit formulas for the $\mathfrak{gl}_{2n}$-modules}
\subsection{Preliminaries}
The following formula will be particularly useful for $m=2$.
\begin{lemma}
\label{glemma}
Let $F:=(e_{j,i})_{i,j}=\sum_{i,j}e_{j,i} \otimes e_{i,j} \in \ca{U(A)} \otimes \ca{A}$.
For any $A,B \in \mathfrak{gl}_{n}$ and for all $m \in \bb{N}$ we have
\[[A,tr(B.F^{m})] = tr ([A,B].F^{m})\]
in $U(\mathfrak{gl}_{n})$ 
\end{lemma}
\begin{proof}
We proceed by induction on $m$. Since $tr(X.F)=X$ the equality clearly holds for $m=1$.
The equation above is linear in both $A$ and $B$ so it suffices to verify it for $A=e_{ij}$, $B=e_{kl}$.
Note that we explicitly have
\[tr(e_{ij}.F^{m+1}) = \sum_{1 \leq r_{1}, \ldots, r_{m}  \leq n} e_{ir_{1}}e_{r_{1}r_{2}} \cdots e_{r_{m}j}.\]
Assume that the equality holds for some fixed $m$. We now compute

\begin{align*}
[e_{ij},&tr(e_{kl}.F^{m+1})] = [e_{ij},\sum_{r_{1}, \ldots, r_{m}} e_{kr_{1}}e_{r_{1}r_{2}} \cdots e_{r_{m}l}]\\
&=\sum_{r_{1}, \ldots, r_{m}}([e_{ij}, e_{kr_{1}}]e_{r_{1}r_{2}} \cdots e_{r_{m}l} + e_{kr_{1}}[e_{ij}, e_{r_{1}r_{2}} \cdots e_{r_{m}l}])\\
&=\sum_{r_{1}, \ldots, r_{m}}(\delta_{jk}e_{ir_{1}} - \delta_{r_{1}i}e_{kj})e_{r_{1}r_{2}} \cdots e_{r_{m}l} + \sum_{r_{1}}e_{kr_{1}}[e_{ij}, \sum_{r_{2}, \ldots r_{m}}e_{r_{1}r_{2}} \cdots e_{r_{m}l}]\\
&=\delta_{jk}tr(e_{il}.F^{m+1}) - e_{kj}\sum_{r_{2}, \ldots, r_{m}}e_{ir_{2}} \cdots e_{r_{m}l} + \sum_{r_{1}}e_{kr_{1}}[e_{ij},tr(e_{r_{1}l}.F^{m})]\\
&=\delta_{jk}tr(e_{il}.F^{m+1}) - e_{kj}tr(e_{il}.F^{m}) + \sum_{r_{1}}e_{kr_{1}}tr([e_{ij},e_{r_{1}l}].F^{m})\\
&=\delta_{jk}tr(e_{il}.F^{m+1}) - e_{kj}tr(e_{il}.F^{m}) + \sum_{r_{1}}e_{kr_{1}}(\delta_{jr_{1}}tr(e_{il}.F^{m})- \delta_{il}tr(e_{r_{1}j}.F^{m}))\\
&=\delta_{jk}tr(e_{il}.F^{m+1}) - e_{kj}tr(e_{il}.F^{m}) + e_{kj}tr(e_{il}.F^{m}) -\delta_{il}\sum_{r_{1}}e_{kr_{1}}tr(e_{r_{1}j}.F^{m}))\\
&=\delta_{jk}tr(e_{il}.F^{m+1}) - \delta_{il}tr(e_{kj}.F^{m+1})\\
&=tr([e_{ij},e_{kl}].F^{m+1}).
\end{align*}
By induction the lemma holds.
\end{proof}

\begin{rmk}
 Fixing $B$ as the identity matrix above we obtain $[A,tr(F^{k})] = 0$ for all $A$ in $\mathfrak{gl}_{n}$ which shows that $tr(F^{k})$ is central in $U(\mathfrak{gl}_{n})$.
In fact, $Z(\mathfrak{gl}_{n}) = \bb{C}[tr(F),tr(F^{2}), \ldots , tr(F^{n})]$. The elements $tr(F^{k})$ are called Gelfand invariants.
\end{rmk}

\subsection{The main result}
We are now ready to state our main result.
Define $\varphi': \ca{A} \rightarrow \ca{U(A)} \otimes \ca{A}$ by 
\[ \varphi': A \mapsto A \otimes I +1 \otimes A.\]
 This is a Lie algebra homomorphism and it
extends to an algebra homomorphism $\varphi: \ca{U(A)} \rightarrow \ca{U(A)} \otimes \ca{A}$.
Also recall that we previously have defined $\psi:  \ca{U(A)} \rightarrow \ca{U(A)} \otimes \ca{A}$ which satisfied $\psi: A \mapsto A \otimes I - 1 \otimes A^{T}$ for $A \in \ca{A}$.
Using these two homomorphisms we now state our main theorem.
\begin{thm}
\label{mainthm}
Define an action of $\mathfrak{gl}_{2n}$ on $M_{I} \simeq U\ca{(A)}$ as follows: for any $a \in U\ca{(A)}$, let
\[ \left( \begin{array}{cc} A & B \\ C & D  \end{array} \right) \cdot a = Aa - aD  +tr(\psi(a).B^{T})  -tr(\varphi(a).F^{2}.C) -tr(\varphi(a).C)tr(F). \tag{1} \]
This is a $\mathfrak{gl}_{2n}$-module structure.
\end{thm}

\begin{proof}
 First, for all $X,Y \in \frak{gl}_{2n}$, $A \in \ca{A}$ and $a \in \ca{U(A)}$ we have

 \begin{align*}
X &\cdot Y \cdot Aa - Y \cdot X \cdot Aa=\\
&=A (X \cdot Y \cdot a) + [XY,A]a - A (Y \cdot X \cdot a) - [YX,A]a\\
&=A (X \cdot Y \cdot a - Y \cdot X \cdot a) + X \cdot [Y,A]a - [X,A]\cdot Ya\\
 & \qquad \qquad - Y \cdot [X,A]a - [Y,A] \cdot Xa\\
&=A \cdot [X,Y] a + [X , [Y,A]]a + [Y,[A,X]]a\\
&= A \cdot [X,Y] a - [A , [X,Y]]a\\
&=[X,Y] \cdot Aa.
\end{align*}

This shows that it suffices to check that 
\[X \cdot Y \cdot 1 - Y \cdot X \cdot 1 = [X,Y]\cdot 1 \]
for all $X,Y \in \frak{gl}_{2n}$ in order to prove that the formula in the theorem gives a module structure. 

We first consider the case $Y:=A_{0} \in \ca{A}$. We compute
\begin{align*}
 &\left( \begin{array}{cc}  A & B \\ C & D  \end{array} \right) \cdot \left( \begin{array}{cc}  A_{0} & 0 \\ 0 & 0  \end{array} \right) \cdot 1
 - \left( \begin{array}{cc}  A_{0} & 0 \\ 0 & 0  \end{array} \right) \cdot \left( \begin{array}{cc}  A & B \\ C & D  \end{array} \right) \cdot 1\\
&= AA_{0}- A_{0}D+tr((A_{0} \otimes I - 1 \otimes A_{0}^{T}).B^{T})-tr((A_{0}\otimes I+1\otimes A_{0}).F^{2}.C)\\
 & \qquad -tr((A_{0}\otimes I+1\otimes A_{0}).C)tr(F)\\
 & \qquad - \big( A_{0}A-A_{0}D+A_{0}tr(B^{T})-A_{0}tr(F^{2}.C)-A_{0}tr(C)tr(F) \big)\\
&=AA_{0}- A_{0}D+A_{0}tr(B^{T})+tr(A_{0}^{T}.B^{T})-A_{0}tr(F^{2}.C) -tr(A_{0}.F^{2}.C)\\
 & \qquad -A_{0}tr(C)tr(F)-tr(A_{0}.C)tr(F)\\
 & \qquad -  A_{0}A+A_{0}D-A_{0}tr(B^{T})+A_{0}tr(F^{2}.C)+A_{0}tr(C)tr(F)\\
&=[A,A_{0}]+tr(A_{0}^{T}.B^{T}) -tr(A_{0}.F^{2}.C)-tr(A_{0}.C)tr(F)\\
&=[A,A_{0}]+tr((A_{0}.B)^{T}) -tr(F^{2}.C.A_{0})-tr(C.A_{0})tr(F)\\
&=\left( \begin{array}{cc}  [A,A_{0}] & A_{0}.B \\ C.A_{0} & 0  \end{array} \right) \cdot 1\\
&=\Big[ \left( \begin{array}{cc}  A & B \\ C & D  \end{array} \right) , \left( \begin{array}{cc}  A_{0} & 0 \\ 0 & 0  \end{array} \right) \Big] \cdot 1.
\end{align*}

It remains to check that $X \cdot Y \cdot 1 - Y \cdot X \cdot 1 = [X,Y] \cdot v$ for $X,Y \in \ca{B,C,D}$.
Moreover, since the right side of $(1)$ is linear in $A,B,C,$ and $D$ it suffices to check $(1)$ it for the standard basis elements of $\mathfrak{gl}_{2n}$.\\

When $X,Y \in \ca{B}$ the calculation is easy:
\begin{align*}
 \left(  \begin{array}{cc} 0 & B \\  0 & 0  \end{array}\right) &\cdot \left(  \begin{array}{cc} 0 & B' \\  0 & 0  \end{array}\right)\cdot 1 
- \left(  \begin{array}{cc} 0 & B' \\  0 & 0  \end{array}\right) \cdot \left(  \begin{array}{cc} 0 & B \\  0 & 0  \end{array}\right) \cdot 1\\
 &= \left(  \begin{array}{cc} 0 & B \\  0 & 0  \end{array}\right) \cdot tr(B') 
- \left(  \begin{array}{cc} 0 & B' \\  0 & 0  \end{array}\right) \cdot tr(B)\\
&=tr(B)tr(B') - tr(B')tr(B) = 0 =  \Big[\left(  \begin{array}{cc} 0 & B \\  0 & 0  \end{array}\right) ,\left(  \begin{array}{cc} 0 & B' \\  0 & 0  \end{array}\right) \Big] \cdot 1.
\end{align*}

Similarly, for $X,Y \in \ca{D}$ we have
\begin{align*}
 \left(  \begin{array}{cc} 0 & 0 \\  0 & D  \end{array}\right) &\cdot \left(  \begin{array}{cc} 0 & 0 \\  0 & D'  \end{array}\right) \cdot 1 
- \left(  \begin{array}{cc} 0 & 0 \\  0 & D'  \end{array}\right) \cdot \left(  \begin{array}{cc} 0 & 0 \\  0 & D  \end{array}\right) \cdot 1\\
 &= -\left(  \begin{array}{cc} 0 & 0 \\  0 & D  \end{array}\right) \cdot D' 
+ \left(  \begin{array}{cc} 0 & 0 \\  0 & D'  \end{array}\right) \cdot D\\
&= D'D - DD'  = [D',D] = \left(  \begin{array}{cc} 0 & 0 \\  0 & [D,D']  \end{array}\right) \cdot 1 = \Big[\left(  \begin{array}{cc} 0 & 0 \\  0 & D  \end{array}\right) ,\left(  \begin{array}{cc} 0 & 0 \\  0 & D'  \end{array}\right) \Big] \cdot 1. 
\end{align*}

For $X \in \ca{B}, Y \in \ca{D}$ we get
\begin{align*}
 \left(  \begin{array}{cc} 0 & B \\  0 & 0  \end{array}\right) &\cdot \left(  \begin{array}{cc} 0 & 0 \\  0 & D  \end{array}\right) \cdot 1 
- \left(  \begin{array}{cc} 0 & 0 \\  0 & D  \end{array}\right) \cdot \left(  \begin{array}{cc} 0 & B \\  0 & 0  \end{array}\right)\cdot 1\\
 &= -\left(  \begin{array}{cc} 0 & B \\  0 & 0  \end{array}\right) \cdot D  - tr(B^{T})\left(  \begin{array}{cc} 0 & 0 \\  0 & D  \end{array}\right) \cdot 1\\
&= -tr((D \otimes I - 1 \otimes D^{T}).B^{T}) +tr(B^{T})D \\
&= -Dtr(B^{T})+tr(D^{T}.B^{T}) +tr(B^{T})D = tr((D.B)^{T})\\
&= \left(  \begin{array}{cc} 0 & D.B \\  0 & 0  \end{array}\right) \cdot 1 = \Big[\left(  \begin{array}{cc} 0 & B \\  0 & 0  \end{array}\right) , \left(  \begin{array}{cc} 0 & 0 \\  0 & D  \end{array}\right) \Big] \cdot 1.
\end{align*}

For $X \in \ca{C}, Y \in \ca{D}$ we apply Lemma~\ref{glemma} for $m=1,2$ to obtain
\begin{align*}
 \left(  \begin{array}{cc} 0 & 0 \\  C & 0  \end{array}\right) &\cdot \left(  \begin{array}{cc} 0 & 0 \\  0 & D  \end{array}\right) \cdot 1 
- \left(  \begin{array}{cc} 0 & 0 \\  0 & D  \end{array}\right) \cdot \left(  \begin{array}{cc} 0 & 0 \\  C & 0  \end{array}\right)v \cdot 1 \\
 &= -\left(  \begin{array}{cc} 0 & 0 \\  C & 0  \end{array}\right) \cdot D  + \left(  \begin{array}{cc} 0 & 0 \\  0 & D  \end{array}\right) \cdot (tr(C.F^{2})+tr(C)tr(F))\\
&= tr((D \otimes I + 1 \otimes D).F^{2}.C) + tr((D \otimes I + 1 \otimes D).C)tr(F)\\
&\qquad- (tr(C.F^{2})+tr(C)tr(F))D\\
&=D \: tr(F^{2}.C) + tr(D.F^{2}.C) + D \: tr(C)tr(F) + tr(D.C)tr(F)\\
&\qquad- (tr(C.F^{2})+tr(C)tr(F))D\\
&=[D,tr(C.F^{2})] + tr(C)[D,tr(F)] + tr(D.F^{2}.C) + tr(D.C)tr(F)\\
&=tr([D,C].F^{2})] + tr(C.D.F^{2}) + tr(D.C)tr(F)\\
&=tr(F^{2}.D.C)] + tr(D.C)tr(F)\\
&=\left(  \begin{array}{cc} 0 & 0 \\  -D.C & 0  \end{array}\right) \cdot 1\\
&= \Big[ \left(  \begin{array}{cc} 0 & 0 \\  C & 0  \end{array}\right) , \left(  \begin{array}{cc} 0 & 0 \\  0 & D  \end{array}\right) \Big] \cdot 1.
\end{align*}

Next, for  $X \in \ca{B}, Y \in \ca{C}$, take $X=e_{i,n+j}$ and $Y=e_{n+k,l}$.
We then have
\begin{align*}
 &e_{i,n+j} \cdot e_{n+k,l}\cdot 1 -  e_{n+k,l}\cdot e_{i,n+j} \cdot 1\\
&= -e_{i,n+j} \cdot (tr(e_{kl}.F^{2})+tr(e_{kl})tr(F)) - e_{n+k,l}\cdot tr(e_{ij}^{T})\\
&= -e_{i,n+j} \cdot ((\sum_{r=1}^{n}e_{kr}e_{rl}) + \delta_{kl}tr(F)) +\delta_{ij}(tr(e_{kl}.F^{2})+tr(e_{kl})tr(F))\\
&= -\sum_{r=1}^{n} \big( tr(\psi(e_{kr}e_{rl}).e_{ji}) \big) - \delta_{kl}tr(\psi(tr(F)).e_{ji}) +\delta_{ij}(tr(e_{kl}.F^{2})+tr(e_{kl})tr(F)) \\
&= -\sum_{r=1}^{n}  tr\big((e_{kr}\otimes I - 1 \otimes e_{rk}).(e_{rl} \otimes I - 1 \otimes e_{lr}).e_{ji} \big) - \delta_{kl}tr((tr(F) \otimes I - 1 \otimes tr(F)).e_{ji})\\
 & \qquad +\delta_{ij}(tr(e_{kl}.F^{2})+tr(e_{kl})tr(F)) \\
&= \sum_{r=1}^{n} \Big( -tr(e_{ji}.e_{rk}.e_{lr}) + e_{kr}tr(e_{ji}.e_{lr}) + e_{rl}tr(e_{ji}.e_{rk}) - e_{kr}e_{rl}tr(e_{ji})\Big)\\
 &\qquad + \delta_{kl}\big( tr(e_{ji}tr(F)) - tr(F)tr(e_{ji}) \big) +\delta_{ij}(tr(e_{kl}F^{2})+tr(e_{kl})tr(F))\\
&= \big(-\delta_{kl}tr(e_{ji}tr(F))+ \delta_{li}e_{kj} + \delta_{jk}e_{il} -\delta_{ji}tr(e_{kl}F^{2})\big)\\
& \qquad +\delta_{kl}\delta_{ji} - \delta_{kl}\delta_{ji}tr(F) +\delta_{ij}tr(e_{kl}F^{2})+\delta_{ij}\delta_{kl}tr(F)\\
&= -\delta_{kl}\delta_{ji}+ \delta_{li}e_{kj} + \delta_{jk}e_{il} +\delta_{kl}\delta_{ji}\\
&= \delta_{li}e_{kj} + \delta_{jk}e_{il} = \delta_{jk}e_{il} -\delta_{li}e_{n+k,n+j} = [e_{i,n+j},e_{n+k,l}] \cdot 1\\
\end{align*}

It remains only to show that $(1)$ holds for $X,Y \in \ca{C}$.
Let $X=e_{n+i,j}$ and $Y=e_{n+k,l}$.
In this case we have
\begin{align*}
e_{n+i,j} &\cdot e_{n+k,l}\cdot 1 -  e_{n+k,l}\cdot e_{n+i,j} \cdot 1\\
=& -e_{n+i,j} \cdot \big( tr(e_{kl}.F^{2}) +tr(e_{kl})tr(F)\big) +  e_{n+k,l}\cdot  \big( tr(e_{ij}.F^{2}) +tr(e_{ij})tr(F)\big)\\
=& -e_{n+i,j} \cdot \big( \sum_{r=1}^{n}e_{kr}e_{rl} +\delta_{kl}tr(F)\big) +  e_{n+k,l}\cdot  \big( \sum_{r=1}^{n}e_{ir}e_{rj} +\delta_{ij}tr(F)\big)\\
=& \Bigg(\sum_{r=1}^{n} \Big( tr(e_{ij}.e_{kr}.e_{rl}.F^{2}) + e_{kr}tr(e_{ij}.e_{rl}.F^{2}) + e_{rl}tr(e_{ij}.e_{kr}.F^{2}) + e_{kr}e_{rl}tr(e_{ij}.F^{2})\\
 &+ \big(tr(e_{ij}.e_{kr}.e_{rl}) + e_{kr}tr(e_{ij}.e_{rl}) + e_{rl}tr(e_{ij}.e_{kr}) + e_{kr}e_{rl}tr(e_{ij})\big) tr(F)\Big)\\
 &+ \delta_{kl}\big( tr(e_{ij}.tr(F).F^{2}) + tr(F)tr(e_{ij}.F^{2}) + tr(e_{ij}.tr(F))tr(F) + tr(F)tr(e_{ij})tr(F) \big) \Bigg)\\
&- \Bigg(\sum_{r=1}^{n} \Big( tr(e_{kl}.e_{ir}.e_{rj}.F^{2}) + e_{ir}tr(e_{kl}.e_{rj}.F^{2}) + e_{rj}tr(e_{kl}.e_{ir}.F^{2}) + e_{ir}e_{rj}tr(e_{kl}.F^{2})\\
 &+ \big(tr(e_{kl}.e_{ir}.e_{rj}) + e_{ir}tr(e_{kl}.e_{rj}) + e_{rj}tr(e_{kl}.e_{ir}) + e_{ir}e_{rj}tr(e_{kl})\big) tr(F)\Big)\\
 &+ \delta_{ij}\big( tr(e_{kl}.tr(F).F^{2}) + tr(F)tr(e_{kl}.F^{2}) + tr(e_{kl}.tr(F))tr(F) + tr(F)tr(e_{kl})tr(F) \big) \Bigg)\\
    =&   n \: tr(e_{ij}.e_{kl}.F^{2}) + e_{kj}tr(e_{il}.F^{2}) +\sum_{r}e_{rl}tr(e_{ij}.e_{kr}.F^{2})  + tr(e_{kl}.F^{2})tr(e_{ij}.F^{2})\\
 &+ \big( n \:tr(e_{ij}.e_{kl}) + e_{kj}tr(e_{il}) + \sum_{r}e_{rl}tr(e_{ij}.e_{kr}) + \delta_{ij}tr(e_{kl}.F^{2})\big) tr(F)\\
 &+ \delta_{kl}\big( tr(e_{ij}.F^{2}) + tr(F)tr(e_{ij}.F^{2}) + \delta_{ij}tr(F) + \delta_{ij}tr(F)tr(F) \big) \\
 &- n \: tr(e_{kl}.e_{ij}.F^{2}) - e_{il}tr(e_{kj}.F^{2}) -\sum_{r}e_{rj}tr(e_{kl}.e_{ir}.F^{2})  - tr(e_{ij}.F^{2})tr(e_{kl}.F^{2})\\
 &+ \big( -n \:tr(e_{kl}.e_{ij}) - e_{il}tr(e_{kj}) - \sum_{r}e_{rj}tr(e_{kl}.e_{ir}) - \delta_{kl}tr(e_{ij}.F^{2})\big) tr(F)\\
 &+ \delta_{ij}\big( -tr(e_{kl}.F^{2}) - tr(F)tr(e_{kl}.F^{2}) - \delta_{kl}tr(F) - \delta_{kl}tr(F)tr(F) \big) \\
    =&  n \delta_{jk} tr(e_{il}.F^{2}) + e_{kj}tr(e_{il}.F^{2}) +\delta_{jk}\sum_{r}e_{rl}tr(e_{ir}.F^{2})  + tr(e_{kl}.F^{2})tr(e_{ij}.F^{2})\\
 &+ n\delta_{jk}\delta_{il}tr(F) + e_{kj}\delta_{il}tr(F) + \delta_{jk}e_{il}tr(F) + \delta_{ij}tr(e_{kl}.F^{2})tr(F)\\
 &+ \delta_{kl}tr(e_{ij}.F^{2}) + \delta_{kl}tr(F)tr(e_{ij}.F^{2}) + \delta_{kl}\delta_{ij}tr(F) + \delta_{kl}\delta_{ij}tr(F)^{2}\\
 &- n \delta_{li} tr(e_{kj}.F^{2}) - e_{il}tr(e_{kj}.F^{2}) -\delta_{li}\sum_{r}e_{rj}tr(e_{kr}.F^{2})  - tr(e_{ij}.F^{2})tr(e_{kl}.F^{2})\\
 &- n\delta_{li}\delta_{kj}tr(F) - e_{il}\delta_{kj}tr(F) - \delta_{li}e_{kj}tr(F) - \delta_{kl}tr(e_{ij}.F^{2})tr(F)\\
 &- \delta_{ij}tr(e_{kl}.F^{2}) - \delta_{ij}tr(F)tr(e_{kl}.F^{2}) - \delta_{ij}\delta_{kl}tr(F) - \delta_{ij}\delta_{kl}tr(F)^{2}\\
    =& \delta_{jk}\sum_{r}e_{rl}tr(e_{ir}.F^{2}) -\delta_{li}\sum_{r}e_{rj}tr(e_{kr}.F^{2}) + e_{kj}tr(e_{il}.F^{2})\\
 &+ n \delta_{jk} tr(e_{il}.F^{2}) + \delta_{kl}tr(e_{ij}.F^{2}) - \delta_{ij}tr(e_{kl}.F^{2}) - n \delta_{li} tr(e_{kj}.F^{2}) - e_{il}tr(e_{kj}.F^{2}) \\
 &+[tr(e_{kl}.F^{2}),tr(e_{ij}.F^{2})] \\
\end{align*}
We proceed to compute $[tr(e_{kl}.F^{2}),tr(e_{ij}.F^{2})]$ separately.

\begin{align*}
 [tr(e_{kl}.F^{2})&,tr(e_{ij}.F^{2})] = \sum_{r}[e_{kr}e_{rl},tr(e_{ij}.F^{2})]\\
&= \sum_{r} \Big( e_{kr}[e_{rl},tr(e_{ij}.F^{2})] + [e_{kr},tr(e_{ij}.F^{2})]e_{rl} \Big)\\
&= \sum_{r}\Big( e_{kr}tr([e_{rl},e_{ij}].F^{2}) + tr([e_{kr},e_{ij}].F^{2})e_{rl} \Big)\\
&= \sum_{r}\Big( \delta_{li}e_{kr}tr(e_{rj}.F^{2}) - \delta_{jr}e_{kr}tr(e_{il}.F^{2}) +\delta_{ri}tr(e_{kj}.F^{2})e_{rl} - \delta_{kj}tr(e_{ir}.F^{2})e_{rl}\Big)\\
&=  \delta_{li} tr(e_{kj}.F^{3})- e_{kj}tr(e_{il}.F^{2}) + tr(e_{kj}.F^{2})e_{il} - \delta_{kj}tr(e_{il}.F^{3})
\end{align*}

Inserting this into the previous expression gives
\begin{align*}
    =&  \delta_{jk}\sum_{r}e_{rl}tr(e_{ir}.F^{2}) -\delta_{li}\sum_{r}e_{rj}tr(e_{kr}.F^{2}) + e_{kj}tr(e_{il}.F^{2})\\
 &+ n \delta_{jk} tr(e_{il}.F^{2}) + \delta_{kl}tr(e_{ij}.F^{2}) - \delta_{ij}tr(e_{kl}.F^{2}) - n \delta_{li} tr(e_{kj}.F^{2}) - e_{il}tr(e_{kj}.F^{2}) \\
 &+ \delta_{li} tr(e_{kj}.F^{3})- e_{kj}tr(e_{il}.F^{2}) + tr(e_{kj}.F^{2})e_{il} - \delta_{kj}tr(e_{il}.F^{3})\\
&= \delta_{jk}\sum_{r}\big(tr(e_{ir}.F^{2})e_{rl} + [e_{rl},tr(e_{ir}.F^{2})]\big) -\delta_{li}\sum_{r} \big(tr(e_{kr}.F^{2})e_{rj} + [e_{rj},tr(e_{kr}.F^{2})] \big)\\
 &+ n \delta_{jk} tr(e_{il}.F^{2}) + \delta_{kl}tr(e_{ij}.F^{2}) - \delta_{ij}tr(e_{kl}.F^{2}) - n \delta_{li} tr(e_{kj}.F^{2}) \\
 &+ \delta_{li} tr(e_{kj}.F^{3}) + [tr(e_{kj}.F^{2}),e_{il}] - \delta_{kj}tr(e_{il}.F^{3})\\
&= \delta_{jk}tr(e_{il}.F^{3}) + \delta_{jk}\sum_{r}tr([e_{rl},e_{ir}].F^{2}) -\delta_{li}tr(e_{kj}.F^{3}) -\delta_{li} \sum_{r}tr([e_{rj},e_{kr}].F^{2})\\
 &+ n \delta_{jk} tr(e_{il}.F^{2}) + \delta_{kl}tr(e_{ij}.F^{2}) - \delta_{ij}tr(e_{kl}.F^{2}) - n \delta_{li} tr(e_{kj}.F^{2}) \\
 &+ \delta_{li} tr(e_{kj}.F^{3}) + tr([e_{kj},e_{il}].F^{2}) - \delta_{kj}tr(e_{il}.F^{3})\\
&= \delta_{jk}\big( \delta_{li}tr(tr(F).F^{2}) - n\: tr(e_{il}.F^{2}) \big) -\delta_{li} \big( \delta_{jk}tr(tr(F).F^{2})-n \: tr(e_{kj}.F^{2}) \big)\\
 &+ n \delta_{jk} tr(e_{il}.F^{2}) + \delta_{kl}tr(e_{ij}.F^{2}) - \delta_{ij}tr(e_{kl}.F^{2}) - n \delta_{li} tr(e_{kj}.F^{2}) \\
 &+ \delta_{ij}tr(e_{kl}.F^{2}) - \delta_{lk}tr(e_{ij}.F^{2})\\
&= \delta_{jk}\delta_{li}tr(F^{2}) -\delta_{li}\delta_{jk}tr(F^{2})=0 = [e_{n+i,j},e_{n+k,l}] \cdot 1\\
\end{align*}
This completes the proof.
\end{proof}

\begin{thm}
\label{corthm}
Define an action of $\mathfrak{gl}_{2n}$ on $M_{Q} \simeq U\ca{(A)}$ as follows: for any $a \in U\ca{(A)}$, let
\[ \left( \begin{array}{cc} A & B \\ C & D  \end{array} \right) \cdot a = Aa - aD  +tr(\psi(a).Q.B^{T})  -tr(\varphi(a).F^{2}.Q^{-T}.C) -tr(\varphi(a).Q^{-T}.C)tr(F).\]
This is a $\mathfrak{gl}_{2n}$-module structure.
\end{thm}
\begin{proof}
For each nonsingular $S \in Mat_{n \times n}$, define $\varphi_{S}: \mathfrak{gl}_{2n} \rightarrow \mathfrak{gl}_{2n} $ by
\[\varphi_{S}: \left( \begin{array}{cc} A & B \\ C & D  \end{array}\right) \mapsto \left( \begin{array}{cc} A & B.S^{-1} \\ S.C & S.D.S^{-1}  \end{array}\right).\]
It is easy to verify that $\varphi_{S}$ is a Lie algebra automorphism and that $\varphi_{S} \circ \varphi_{T} = \varphi_{S.T}$, so the map $\Xi: Mat_{n \times n}(\bb{C})^{*} \rightarrow Aut(\mathfrak{gl}_{2n})$ with
 $S \mapsto \varphi_{S}$ is an injective algebra homomorphism.
Let $V$ be the $\frak{gl}_{2n}$ module from in Theorem~\ref{mainthm}. Now by the action of $\mathfrak{gl}_{2n}$ on the twisted module $V_{Q}:={}^{\varphi_{Q^{-T}}}V$ is precisely as in the statement of this theorem.
\end{proof}

The modules $V_{Q}$ now satisfy the conditions of Theorem~\ref{introthm} in the introduction:
\begin{proof}
(of Theorem~\ref{introthm})
The module $V_{Q}$ is simple since $Res_{\ca{A+B}}^{\frak{gl}_{2n}}V_{Q} \simeq M_{Q}$ is. That the GK-dimension is $n^{2}$ and that $Res_{\ca{A}}^{\frak{gl}_{2n}} V_{Q} \simeq \ca{U(A)}$
 follows directly from the definition in Theorem~\ref{corthm}. Since the linear maps $tr(\psi(-).B^{T}): \ca{U(A)} \rightarrow \ca{U(A)}$ never increases the degree of a monomial,
the module $Res_{\ca{B}}^{\frak{gl}_{2n}}V_{Q}$ is locally finite. The fourth point follows from similar arguments: the maps $tr(\psi(-).F^{2}.C): \ca{U(A)} \rightarrow \ca{U(A)}$
have degree $2$ and the maps $A(-)$ and $(-)D$ clearly have degree $1$ (compare with Theorem~\ref{corthm}).
Finally, we note that any isomorphism $\varphi:V_{Q} \rightarrow V_{Q'}$ must map the generator of $V_{Q}$ to a multiple of the generator of $V_{Q'}$.
But then $q_{ij}'\varphi(1) = e_{i,n+j} \varphi(1) = \varphi(e_{i,n+j} \cdot 1) = q_{ij}\varphi(1)$, showing that $Q=Q'$ whenever such an isomorphism exists.
\end{proof}

\subsection{Alternative formula}
Since the automorphisms $\varphi$ and $\psi$ themselves are not very explicit, we present another formula for how elements of $\frak{gl}_{2n}$ act on monomials of $\ca{U(A)}$.
We need some more conventions in notation for this formula.

In the argument of the trace functions, any product is by convention to be taken in $Mat_{n \times n}(U(\mathfrak{gl}_{n}))$ 
(in particular we identify $\ca{A}$ with $Mat_{n \times n}(\bb{C})$ here). Outside the trace function all products are in $U(\mathfrak{gl}_{n})$.
When $S \subset \bb{Z}$, the product $\prod_{i \in S}A_{i}$ means that the product is to be taken in order inherited from $\bb{Z}$. For example, $\prod_{i \in \{3,2,5\}}A_{i} = A_{2}A_{3}A_{5}$.
For $S \subset \{1, \ldots, k\}$, we denote by $S^{*}$ the complement $\{1, \ldots, k\} \setminus S$ and by $|S|$ the cardinality of $S$.

\begin{thm}
Let $a = \prod_{i=1}^{k}A_{i}$ be a monomial in $V_{Q}$ (see Theorem~\ref{corthm}). The action of $\frak{gl}_{2n}$ on the monomial $a$ can be written explicitly as follows.
\begin{align*}
&\left( \begin{array}{cc}
   A & B \\
   C & D
  \end{array} \right) \cdot \prod_{i=1}^{k}A_{i} := A \prod_{i=1}^{k}A_{i} - \prod_{i=1}^{k}A_{i}(Q^{-T}.D.Q^{T}) \\
 &+ \sum_{S \subset \{1, \ldots, k\}}\Big(  \prod_{i \in S^{*}} A_{i} \Big) \Big( (-1)^{|S|}tr(B^{T}. \prod_{i \in S} A_{i}^{T}.Q) -   tr( Q^{-T}.C.\prod_{i \in S} A_{i}. F^{2})  - tr( Q^{-T}.C.\prod_{i \in S} A_{i}) tr(F) \Big)
\end{align*}
\end{thm}
\begin{proof}
 This follows by induction on $k$ by comparing with the formula in Theorem~\ref{corthm}. The verification is omitted here.
\end{proof}

\noindent Department of Mathematics, Uppsala University, Box 480, SE-751 06, Uppsala, Sweden, email: jonathan.nilsson@math.uu.se

\end{document}